\documentclass[11pt]{article}
\usepackage[dvips]{graphicx,psfrag}
\usepackage{psfragx}
\usepackage[english]{babel}
\usepackage{colortbl}
\usepackage{mathrsfs}
\usepackage{multicol}
\usepackage{amsmath,amsfonts,amsthm}
\usepackage{dsfont}
\usepackage[T1]{fontenc}
\usepackage{ae}
\usepackage{listings}
\usepackage{epsfig,epsf}
\usepackage[latin1]{inputenc}
\usepackage{color}
\usepackage{xcolor}
\usepackage{listings}
\usepackage{proba}

\providecommand{\keywords}[1]{\textbf{Keywords:} #1}

\definecolor{listinggray}{gray}{0.9}
\definecolor{lbcolor}{rgb}{0.9,0.9,0.9}
\newtheorem{teo}{Theorem}[section]
\newtheorem{lema}{Lemma}[section]

\newtheorem{coro}{Corollary}[section]

\newcommand{\e}{\mathds{E}}

\newcommand{\p}{\mathds{P}}

\begin{document}



\title{Bounds for a binomial sum involving powers of the summation index}
\author{Eliardo G. Costa\\
  Departamento de Estat\'istica, Universidade de S\~ao Paulo, S\~ao Paulo, Brazil\\
  \texttt{eliardo@ime.usp.br}}

\maketitle







\abstract{Recently, the properties of a binomial sum related to the multi-link inverted pendulum enumeration problem have been studied. In this note, we establish bounds for this binomial sum.}\\



\keywords{multi-link inverted pendulum, random variable, expectation value.}


\section{Introduction}
The binomial sum $S_p(n)$ defined by
\begin{equation*}
S_p(n)=\sum_{j=1}^nj^p{n+j\choose j}
\end{equation*}

\noindent
has been studied in the literature because it is related to the multi-link inverted pendulum enumeration problem and it is important to know its properties (see \cite{ParisLarcombe2012} and references therein, for example). In this note, we provide bounds for $S_p(n)$ with $p$ a real positive number using a stochastic approach. Throughout this note $\p(.)$ and $\e[.]$ will denote the probability and expectation operator, respectively. 

\section{Results}
The key of the problem to obtain the bounds is to realize that the sum $S_p(n)$ may be write as a expectation value of a random variable with some adjustments, and using the following Lemma we obtain the result. 
\begin{lema}\label{bounds}
Let $g(x)>0$ be an even function and nondecreasing on $[0,\infty)$. Suppose that $\e[g(X)]<\infty$. Then for any $x>0$
\begin{equation*}
\frac{\e[g(X)]-g(x)}{a.s.sup\ g(X)}\leq\p(|X|\geq x)\leq\frac{\e[g(X)]}{g(x)},
\end{equation*}

\noindent
where $a.s.sup\ g(X)=\inf\{t:\p(g(X)>t)=0\}$.
\end{lema}
\begin{proof}
see \cite[~pg. 52]{ProbIneq}.
\end{proof}

\begin{teo}\label{S-bounds}
For $p>0$, we have $n^p{2n\choose n}\leq S_p(n)\leq \left[n^p\frac{2n+1}{n+1}+n\right]{2n\choose n}$.
\end{teo}
\begin{proof}
Consider a random variable $X$ with probability distribution
\begin{equation}\label{prob-fun}
\p(X=j)=c^{-1}{n+j\choose j},\quad j=0,1,\ldots,n,
\end{equation}

\noindent
where $c={2n+1\choose n}$. It is easy to see that (\ref{prob-fun}) is a probability function since $c=S_0(n)+1$ (see \cite[~pg. 159]{Grahametal}, for example). Then we have
\begin{equation*}
\p(|X|\geq n)=\p(X=n)=c^{-1}{2n\choose n}.
\end{equation*}

\noindent
Using Lemma \ref{bounds} with $g(x)=|x|^p$, $p>0$, we have
\begin{equation*}
S_p(n)=c\e[|X|^p]\geq cn^p\p(|X|\geq n)=n^p{2n\choose n}.
\end{equation*}

\noindent
In this case $a.s.sup\ |X|^p=\inf\{t:\p(|X|^p>t)=0\}=n$. Using Lemma \ref{bounds} we have
\begin{equation*}
\e[|X|^p]-n^p\leq n\p(|X|\geq n)=nc^{-1}{2n\choose n},
\end{equation*}

\noindent
which implies that $\e[|X|^p]\leq n^p+nc^{-1}{2n\choose n}$. Then, we obtain
\begin{equation*}
S_p(n)=c\e[|X|^p]\leq cn^p+n{2n\choose n}=n^p{2n+1\choose n}+n{2n\choose n}.
\end{equation*}
%
\noindent
Since ${2n+1\choose n}=\frac{2n+1}{n+1}{2n\choose n}$ we may write the upper bound as
\begin{equation*}
\left[n^p\frac{2n+1}{n+1}+n\right]{2n\choose n}.
\end{equation*}

\end{proof}

\begin{teo}\label{M-bounds}
Let $M_p(n)=S_p(n)/[n^p{2n\choose n}]$. Then the following statements hold as $n\to\infty$:
\begin{enumerate}
\item $1\leq M_p(n)\leq2$, if $p>1$;
\item $1\leq M_1(n)\leq3$.
\end{enumerate}
\end{teo}
\begin{proof}
Using Theorem \ref{S-bounds} we have
\begin{equation*}
1\leq M_p(n)\leq \frac{2n+1}{n+1}+\frac{n}{n^p}.
\end{equation*}

\noindent
The limit of upper bound is 2 if $p>1$ and is 3 if $p=1$, as $n$ approaches infinity.

\end{proof}

\begin{lema}\label{bounds-2n-n}
If $n>2$, then $\frac{4^{n}}{2\sqrt{n}}<{2n\choose n}<\frac{(2n+2)^n}{(n+1)!}$.
\end{lema}
\begin{proof}
see \cite[~pg. 132]{Sierpinski} and \cite[~pg. 297]{BartosZnam}.
\end{proof}

\begin{coro}\label{S-simple-bounds}
If $n>2$, then $\frac{n^p4^n}{2\sqrt{n}}<S_p(n)<\left[n^p\frac{2n+1}{n+1}+n\right]\frac{(2n+2)^n}{(n+1)!}$.
\end{coro}
\begin{proof}
The result follows from Theorem \ref{S-bounds} and Lemma \ref{bounds-2n-n}.
\end{proof}

\section{Concluding remarks}
Theorem \ref{S-bounds} and Corollary \ref{S-simple-bounds} provide lower and upper bounds for $S_p(n)$ with $p$ a real positive number. The bounds presented in Corollary \ref{S-simple-bounds} are less refined than the ones presented in Theorem \ref{S-bounds}, but this result does not present combinatorial numbers. Theorem \ref{M-bounds} is in agreement with asymptotic results presented in \cite{ParisLarcombe2012}.

\section{Acknowledgements}
The author would like to thank the Coordena\c c\~ao de Aperfei\c coamento de Pessoal de N\'ivel Superior (CAPES, Brazil) and the Conselho Nacional de Desenvolvimento Cient\' \i fico e Tecnol\' ogico (CNPq, Brazil, grant 133211/2011-8) for partial financial support.










\end{document}